\documentclass{amsart}

\usepackage{bbm}
\usepackage{nicefrac}
\usepackage{hyperref}

\newcommand*{\mailto}[1]{\href{mailto:#1}{\nolinkurl{#1}}}
\newcommand{\arxiv}[1]{\href{http://arxiv.org/abs/#1}{arXiv:#1}}

\newtheorem{theorem}{Theorem}[section]

\newtheorem{lemma}[theorem]{Lemma}

\newtheorem{corollary}[theorem]{Corollary}
\newtheorem{remark}[theorem]{Remark}

\newcommand{\R}{{\mathbb R}}

\newcommand{\C}{{\mathbb C}}

\newcommand{\be}{\begin{equation}}
\newcommand{\ee}{\end{equation}}

\newcommand{\ol}{\overline}
\newcommand{\ti}{\tilde}

\newcommand{\spr}[2]{\langle #1 , #2 \rangle}

\newcommand{\im}{\mathrm{Im}}

\newcommand{\dom}[1]{\mathrm{dom}\left(#1\right)}

\newcommand{\lam}{\lambda}
\newcommand{\gam}{\gamma}

\numberwithin{equation}{section}

\begin{document}

\title[Uniqueness for Inverse Sturm--Liouville Problems]{Uniqueness for Inverse Sturm--Liouville Problems with a Finite Number of Transmission Conditions}

\author[M.\ Shahriari]{Mohammad Shahriari}
\address{Faculty of Mathematical Sciences\\ University of Tabriz\\ Tabriz 51664\\ Iran\\ and Faculty of Mathematics\\ University of Vienna\\
Nordbergstrasse 15\\ 1090 Wien\\ Austria}
\email{\mailto{shahriari@tabrizu.ac.ir}}

\author[A. Jodayree Akbarfam]{Aliasghar Jodayree Akbarfam}
\address{Faculty of Mathematical Sciences\\ University of Tabriz\\ Tabriz 51664\\ Iran}
\email{\mailto{akbarfam@yahoo.com}}

\author[G.\ Teschl]{Gerald Teschl}
\address{Faculty of Mathematics\\ University of Vienna\\
Nordbergstrasse 15\\ 1090 Wien\\ Austria\\ and
International Erwin Schr\"odinger Institute for Mathematical Physics\\
Boltzmanngasse 9\\ 1090 Wien\\ Austria}
\email{\mailto{Gerald.Teschl@univie.ac.at}}
\urladdr{\url{http://www.mat.univie.ac.at/~gerald/}}

\thanks{J. Math. Anal. Appl. {\bf 395}, 19--29 (2012)}
\thanks{\it Research supported by the Austrian Science Fund (FWF) under Grant No.\ Y330}


\begin{abstract}
We establish various uniqueness results for inverse spectral problems of Sturm--Liouville operators with
a finite number of discontinuities at interior points at which we impose the usual transmission conditions.
We consider both the cases of classical Robin and of eigenparameter dependent boundary conditions.
\end{abstract}

\maketitle

\section{Introduction}
\label{sec:int}

In the seminal paper \cite{hal}, Hald, motivated by the inverse problem for the torsional modes of the earth,
investigated Sturm--Liouville problems with a discontinuity at an interior point. Hald proved a Hochstadt--Liebermann
result in the case of one transmission condition which was later on extended to two transmission conditions by
Willis \cite{wil}. Moreover, Kobayashi \cite{kob} proved a similar result in the case for problems with a reflection symmetry.
More recently, Mukhtarov, Kadakal and Muhtarov \cite{mkm} and two of us \cite{sj} have investigated the case with one transmission
condition and eigenparameter dependent boundary conditions, and derived asymptotic formulas for the eigenvalues and eigenfunctions.
Even more recently, these results were extended to two and three transmission conditions in \cite{km} and \cite{sen}, respectively.
The purpose of the present paper is to show how to handle an arbitrary finite number of transmission conditions and to use
the asymptotic formulas to prove several uniqueness results. In particular, we will introduce a Weyl $m$-function which
uniquely determines the parameters of the problem. We also show that this Weyl function is a meromorphic Herglotz--Nevanlinna
function which is uniquely determined by its poles and residues, as well as by its poles and zeros. In particular, we also
obtain a two spectra result. This generalizes the results of Amirov \cite{am} in the case of one transmission condition to the case of a finite number of transmission and eigenparameter dependent boundary conditions. Moreover, we will also generalize the Hochstadt--Liebermann type result from Hald to the present situation.

To the best of our knowledge, this is the first result
concerning more than three transmission conditions. In particular, it was necessary to modify the usual arguments at several
places in order to make up for some key estimates which cannot be easily shown in the present situation (cf.\ the intricate nature of
the high energy asymptotics of solutions in Theorem~\ref{thm:asym}).
For related results, we refer to \cite{adm}, \cite{et2}, \cite{sy}, \cite{wang}, \cite{y}, \cite{yy}.

Sturm--Liouville problems with transmission conditions at interior points arise in a variety of applications in engineering
and we refers to \cite{am} for a nice discussion and further information. Here we only want to mention that they also appear
in the description of delta interactions (which play an important role in quantum mechanics \cite{aghh}) and of radially symmetric
quantum trees (cf. the discussion in Section~4 of \cite{sst} and the references therein). For general background on inverse Sturm--Liouville problems we refer (e.g.) to the monographs \cite{fy}, \cite{lev}, \cite{te}.

We will first start with the usual Robin boundary conditions and then briefly show how to extend the present approach to the
more general case of eigenparameter dependent boundary conditions in our last section.

\section{The Hilbert space formulation and properties of the spectrum}

In the first part of our paper we consider the boundary value problem
\be\label{1}
   \ell y:=-y''+qy=\lambda y
\ee
subject to the Robin boundary conditions
\begin{align}\label{2}
    &L_1(y):= y'(0) + h\, y(0)= 0, \nonumber\\
   & L_2(y):= y'(\pi)+H\, y(\pi) =0
\end{align}
with transmission (discontinuous) conditions
\begin{align}\label{3}
  U_{i}(y)&:= y(d_i+0)- a_iy(d_i-0)=0, \nonumber \\
  V_{i}(y)&:=y'(d_i+0)-b_iy'(d_i-0)-c_i y(d_i-0)=0,\
\end{align}
where $q(x)$ is real-valued function in $L^1[0,\pi]$. We also assume that $h$, $H$ and
$a_i$, $b_i$, $c_i$ $d_i$, $i=1,2,\dots,m-1$ (with $m\geq 2$) are real numbers, satisfying
$a_i b_i>0$, $d_0=0<d_1<d_2<...<d_{m-1}<d_m=\pi$.
For simplicity we use the notation $L=L(q(x); h;H;d_i)$, for the problem \eqref{1}--\eqref{3}.

To obtain a self-adjoint operator we introduce the following weight
function
\be
w(x) =
\begin {cases}
1, & 0 \leq x < d_1,\\
\frac{1}{a_1 b_1}, & d_1 <x <d_2,\\
\vdots\\
\frac{1}{a_1 b_1 \cdots a_{m-1} b_{m-1}}, & d_{m-1} <x \leq\pi.
\end{cases}
\ee
Now our Hilbert space will be $\mathcal{H}:=L_2((0,\pi); w)$ associated with the weighted
inner product
 \begin{align}
    &\spr{f}{g}_{\mathcal{H}}:=\int_0^\pi f\ol{g} w.
 \end{align}
The corresponding norm will be denoted by
$\|f\|_{\mathcal{H}}=\spr{f}{f}_{\mathcal{H}}^{1/2}$.
In this Hilbert space we construct the operator
\be
    A:\mathcal{H}\rightarrow \mathcal{H}
\ee
with domain
\be
    \dom{A}=\left\{ f\in\mathcal{H} \left|
    \begin{array}{c}
         f, f'\in AC\big(\cup_0^{m-1}(d_i,d_{i+1})\big), \\
             \ell f\in L^2(0,\pi), \: U_i(f)=V_i(f)=0
           \end{array}\right.\right\}
\ee
by
\[
A f= \ell f \quad \text{with} f\in \dom{A}.
\]
Throughout this paper $AC\big(\cup_0^{m-1}(d_i,d_{i+1})\big)$ denotes the set of all functions whose restriction to $(d_i,d_{i+1})$ is
absolutely continuous for all $i=0,\dots,m-1$. In particular, those functions will have limits at the boundary points $d_i$.

\begin{lemma}
The operator $A$ is self-adjoint.
\end{lemma}

In particular, the eigenvalues of $A$, and hence of $L$, are real and simple. To see that they are simple
it suffices to observe that the associated Cauchy problem \eqref{1}, \eqref{3} subject to the initial conditions
 $f(x_0\pm0)=f_0$, $f'(x_0\pm0)=f_1$ (with $x_0\in[0,1]$) has a unique solution.

For any function $f\in\dom{A}$ we will denote by $f_j$, $1\le j\le m$, the restriction of $f$ to the subinterval $(d_{j-1},d_j)$. Moreover,
we will set $f_j(d_{j-1})=f(d_{j-1}+0)$ and $f_j(d_j)=f(d_j-0)$.

Suppose that the functions $\varphi(x,\lambda)$ and $\psi(x,\lambda)$ are solutions of \eqref{1} under the initial conditions
\begin{align}\label{010}
      \varphi(0,\lambda)= 1 ,\ \varphi'(0,\lambda)= -h,
\end{align}
and
\begin{align}\label{0101}
      \psi(\pi,\lambda)=1,\ \psi'(\pi,\lambda)= -H
\end{align}
as well as the jump conditions \eqref{3}, respectively.
It is easy to see that equation \eqref{1} under the initial conditions \eqref{010} or \eqref{0101} has a unique solution $\varphi_1(x,\lambda)$ or
$\psi_m(x, \lambda)$, which is an entire function of $\lambda\in\C$ for each fixed point $x\in [0,d_1)$ or $x\in(d_{m-1},\pi]$. From
the linear differential equations we obtain that the modified Wronskian
\begin{align}
   W(u,v)= w(x) \big( u(x) v'(x) - u'(x) v(x) \big)
\end{align}
is constant on $x\in[0,d_1)\cup_1^{m-2}(d_i,d_i+1)\cup(d_{m-1},\pi]$ for two solutions $\ell u =\lambda u$, $\ell v =\lambda v$ satisfying the transmission conditions \eqref{3}. Moreover, we set
 \be\label{d3}
\Delta(\lambda):= W(\varphi(\lambda),\psi(\lambda)) = L_1(\psi(\lambda)) = - w(\pi) L_2(\varphi(\lambda)).
\ee
Then $\Delta(\lambda)$ is an entire function whose roots $\lambda_n$ coincide with the eigenvalues of $L$.
Moreover, the eigenfunctions $\varphi_i(x,\lambda_n)$ and $\psi_i(x,\lambda_n)$
 associated with a certain eigenvalue $\lambda_n$, satisfy the relation $\psi_i(x,\lambda_n)=\beta_n \varphi_i(x,\lambda_n)$,
 where, by \eqref{010},
\be
\beta_n= \psi(0,\lambda_n).
\ee
We also define the norming constant by
\[
\gamma_n:=\|\varphi(x,\lambda_n)\|_{\mathcal{H}}^{-2}.
\]
Then it is straightforward to verify:

\begin{lemma}\label{lem:deldot}
All zeros $\lambda_n$ of $\Delta(\lambda)$ are simple and the derivative is given by
\be
\dot{\Delta}(\lambda_n)=-\gamma_n^{-1}\beta_n.
\ee
\end{lemma}

Finally, we point out a simple unitary transformation for our eigenvalue problem which is easy
to check:

\begin{lemma}\label{lem:unitary}
The map
\[
U: \mathcal{H} \to \hat{\mathcal{H}}= L_2(0,\pi), \quad f(x)  \mapsto \hat{f}(x) =\sqrt{w(x)} f(x)
\]
maps $A$ unitarily to $\hat{A}$ associated with $\hat{a}_i = (a_i/b_i)^{1/2}$, $\hat{b}_i = (b_i/a_i)^{1/2}$, $\hat{c}_i= c_i (a_i b_i)^{-1/2}$
and all remaining items unchanged. In particular, $\hat{a}_i \hat{b}_i=1$ and hence $\hat{w}(x)=1$.
\end{lemma}

\begin{remark}
After a similar transformation as above we can assume that $a_i=1$ without loss of generality and then
our operator is a special case of a measure-valued Sturm--Liouville operator \cite{et1}
\[
\ell y(x)=  \frac{1}{w(x)} \frac{d}{dx} \left( - w(x) y'(x) + \int_0^x y(t) d\chi(t) \right)
\]
associated with the measure-valued potential
\[
d\chi(x) = q(x)dx+\sum_{i=1}^{m-1} w(d_i+) c_i \delta_{d_i}(x),
\]
where $\delta_d$ is the Dirac delta measure located at $d$.
\end{remark}

\section{Asymptotic form of solutions and eigenvalues}

\begin{theorem}\label{thm:asym}
Let $\lambda=\rho^2$ and $\tau:=\im\rho$. For equation \eqref{1} with boundary conditions
 \eqref{2} and jump conditions \eqref{3} as $|\lambda|\rightarrow\infty$, the following
 asymptotic formulas hold:
\begin{align}\label{12}
 \varphi(x,\lambda)=
  \begin{cases}
           \cos\rho x +O(\frac{\exp(|\tau|x)}{\rho}),  \qquad 0\leq x<d_1,  \\
            \alpha_1\cos\rho x +\alpha'_1\cos\rho(x-2d_1)+O(\frac{\exp(|\tau|x)}{\rho}), \qquad d_1<x<d_2,\\
            \alpha_1\alpha_2\cos\rho x +\alpha'_1\alpha_2\cos\rho(x-2d_1)+\alpha_1\alpha'_2\cos\rho(x-2d_2)\\
   \quad +\alpha'_1\alpha'_2\cos\rho(x+2d_1-2d_2)+O(\frac{\exp(|\tau|x)}{\rho}), \qquad d_2<x<d_3,\\
   \quad \vdots\\
         \alpha_1\alpha_2...\alpha_{m-1}\cos\rho x+\\
  \quad + \alpha'_1\alpha_2...\alpha_{m-1}\cos\rho(x-2d_1)+...\\
  \quad + \alpha_1\alpha_2...\alpha'_{m-1} \cos\rho(x-2d_{m-1})+\\
  \quad + \alpha'_1\alpha'_2\alpha_3...\alpha_{m-1}\cos\rho(x+2d_1-2d_2)+...\\
  \quad +\alpha_1...\alpha'_i...\alpha'_j...\alpha_{m-1}\cos\rho(x+2d_i-2d_j)\\
  \quad+ \alpha_1...\alpha'_i...\alpha'_j...\alpha'_k...\alpha_{m-1}\cos\rho(x-2d_i+2d_j-2d_k)+... \\
  \quad + \alpha'_1\alpha'_2...\alpha'_{m-1}\cos\rho(x+2(-1)^{m-1} d_1+2(-1)^{m-2}d_2+...-2d_{m-1})\\
  \quad +O(\frac{\exp(|\tau|x)}{\rho}), \qquad  d_{m-1}<x\leq\pi,
\end{cases}
\end{align}
and
\begin{align}\label{12'}
  \varphi'(x,\lambda)=
  \begin{cases}
         \rho[-\sin\rho x] +O(\exp(|\tau|x)),  \qquad 0\leq x<d_1,\\
          \rho[-\alpha_1\sin\rho x -\alpha'_1\sin\rho(x-2d_1)]+O(\exp(|\tau|x)), \qquad d_1<x<d_2,\\
          \rho[-\alpha_1\alpha_2\sin\rho x -\alpha'_1\alpha_2\sin\rho(x-2d_1)-\\
   \qquad -\alpha_1\alpha'_2\sin\rho(x-2d_2) -\alpha'_1\alpha'_2\sin\rho(x+2d_1-2d_2)]\\
   \qquad +O(\exp(|\tau|x)), \qquad d_2<x<d_3,\\
                       \quad \vdots&\\
      \rho[ -\alpha_1\alpha_2...\alpha_{m-1}\sin\rho x-\alpha'_1\alpha_2...\alpha_{m-1}\sin\rho(x-2d_1)-...-\alpha_1\alpha_2...\alpha'_{m-1}\\
 \quad \sin\rho(x-2d_{m-1})-\alpha'_1\alpha'_2\alpha_3...\alpha_{m-1}\sin\rho(x+2d_1-2d_2)-...\\
\quad -\alpha_1...\alpha'_i...\alpha'_j...\alpha_{m-1}\sin\rho(x+2d_i-2d_j)\\
\quad -\alpha_1...\alpha'_i...\alpha'_j...\alpha'_k...\alpha_{m-1}\sin\rho(x-2d_i+2d_j-2d_k)+...\\
\quad -\alpha'_1\alpha'_2...\alpha'_{m-1}\sin\rho(x+2(-1)^{m-1}d_1+2(-1)^{m-2}d_2+...-2d_{m-1})]\\
\quad +O(\exp(|\tau|x)), \qquad  d_{m-1}<x\leq\pi,
\end{cases}
\end{align}
where
\be\label{a}
\alpha_i=\frac{a_i+b_i}{2} \ \ \text{and} \ \ \alpha'_i=\frac{a_i-b_i}{2},
\ee
 for $i=1,2,...,m-1$. The characteristic function satisfies
\begin{align}
\Delta(\lambda)=&\rho w(\pi)[\alpha_1\alpha_2...\alpha_{m-1}\sin\rho \pi+\alpha'_1\alpha_2...\alpha_{m-1}\sin\rho(\pi-2d_1)+...+\alpha_1\alpha_2...\alpha'_{m-1}\nonumber\\ &\sin\rho(\pi-2d_{m-1})+\alpha'_1\alpha'_2\alpha_3...\alpha_{m-1}\sin\rho(\pi+2d_1-2d_2)+...\nonumber\\ &+\alpha_1...\alpha'_i...\alpha'_j...\alpha_{m-1}\sin\rho(\pi+2d_i-2d_j)\nonumber\\
&+\alpha_1...\alpha'_i...\alpha'_j...\alpha'_k...\alpha_{m-1}\sin\rho(\pi-2d_i+2d_j-2d_k)+...\nonumber\\
&+\alpha'_1\alpha'_2...\alpha'_{m-1}\sin\rho(\pi+2(-1)^{m-1}d_1+2(-1)^{m-2}d_2+...-2d_{m-1})]\nonumber\\
&+O(\exp(|\tau|\pi)).
\end{align}
\end{theorem}

\begin{proof}
Suppose $C(x,\lambda)$ and $S(x,\lambda)$ are the cosine and sine-type solutions of \eqref{1} with jump
conditions \eqref{3} corresponding to the initial conditions
\[
C(0,\lambda)=1,\ C'(0,\lambda)=0\ \text{and}\  S(0,\lambda)=0,\ S'(0,\lambda)=1.
\]
First of all observe
 \begin{align*}
C(x,\lambda)=
 \begin{cases}
                \cos\rho x+O(\frac{\exp|\tau|x}{\rho}), \qquad 0\leq x<d_1, \\
                 a_1C_1(d_1,\lambda)\cos\rho (x-d_1)+\frac{b_1}{\rho}C'_1(d_1,\lambda)\sin\rho(x-d_1)\\
                 \qquad+O(\frac{\exp|\tau|(x-d_1)}{\rho}), \qquad d_1< x<d_2,  \\
                 a_2C_2(d_2,\lambda)\cos\rho (x-d_2)+\frac{b_2}{\rho}C'_2(d_2,\lambda)\sin\rho(x-d_2)\\
                 \qquad+O(\frac{\exp|\tau|(x-d_2)}{\rho}), \qquad d_2< x<d_3, \\
                 \quad \vdots &\\
                a_{m-1}C_{m-1}(d_{m-1},\lambda)\cos\rho (x-d_{m-1}) +\\
                \qquad + \frac{b_{m-1}}{\rho} C'_{m-1}(d_{m-1},\lambda)\sin\rho(x-d_{m-1}) +\\
                \qquad+O(\frac{\exp|\tau|(x-d_{m-1})}{\rho}), \qquad d_{m-1}< x\leq\pi.\\
        \end{cases}
\end{align*}
Next we substitute the $i$'th statement into the $(i+1)$'th statement to obtain
 \begin{align*}
C(x,\lambda)=
\begin{cases}
                \cos\rho x+O(\frac{\exp|\tau|x}{\rho}), \quad 0\leq x<d_1, \\
                 \alpha_1\cos\rho (x)+\alpha'_1\cos\rho(x-2d_1)+O(\frac{\exp|\tau|x}{\rho}), \quad d_1< x<d_2,  \\
                 \alpha_1\alpha_2\cos\rho (x)+\alpha'_1\alpha_2\cos\rho (x-2d_1)+\alpha_1\alpha'_2\cos\rho (x-2d_2)\\
                 \quad +\alpha'_1\alpha'_2\cos\rho (x+2d_1-2d_2)+O\left(\frac{\exp|\tau|x}{\rho}\right), \quad d_2< x<d_3, \\
                 \quad \vdots\\
                \alpha_1\alpha_2...\alpha_{m-1}\cos\rho x+\alpha_1...\alpha'_i...\alpha_{m-1}\cos\rho(x-2d_i)\\
 \quad +\alpha_1...\alpha'_i...\alpha'_j...\alpha_{m-1}\cos\rho(x+2d_i-2d_j)\\
\quad +\alpha_1...\alpha'_i...\alpha'_j...\alpha'_k...\alpha_{m-1}\cos\rho(x-2d_i+2d_j-2d_k)+...+\\
\quad +\alpha'_1\alpha'_2...\alpha'_{m-1}\cos\rho(x+2(-1)^{m-1}d_1+2(-1)^{m-2}d_2+...-2d_{m-1})\\
\quad +O\left(\frac{\exp|\tau|x}{\rho}\right), \quad  d_{m-1}<x\leq\pi,
\end{cases}
\end{align*}
where $\alpha_i$ and $\alpha'_i$ is defined in \eqref{a} and $i<j<k$, $i,j,k=1,2,...,m-1$.
Similar calculations establish the asymptotic form of $C'(x,\lambda)$, $S(x,\lambda)$,  and $S'(x,\lambda)$. This
proves the theorem upon observing $\varphi(x,\lambda)=C(x,\lambda)+h\,S(x,\lambda)$.
\end{proof}

It follows from the above theorem that
\be
|\varphi^{(\nu)}(x,\lambda)|=O(|\rho|^{\nu}\exp(|\tau|x)),\ \ 0\leq x \leq \pi, \: \nu=0,1.
\ee
By changing $x$ to $\pi-x$ one can obtain the asymptotic form of $\psi(x,\lambda)$ and $\psi'(x,\lambda)$. In particular,
\be\label{p1}
|\psi^{(\nu)}(x,\lambda)|=O(|\rho|^{\nu}\exp(|\tau|(\pi-x))),\ \ 0\leq x \leq \pi, \: \nu=0,1.
\ee

As a consequence of Valiron's theorem (\cite[Thm.~13.4]{levin}) we obtain:

\begin{theorem}
The eigenvalues $\lambda_n = \rho_n^2$ of the boundary value problem $L$ satisfy
\[
{\rho_n}= n+o(n)
\]
as $n\rightarrow\infty$.
\end{theorem}


\section{Uniqueness results for Robin boundary conditions}

In this section we investigate the inverse problem of the reconstruction
of a boundary value problem $L$ from its spectral characteristics. We consider
three statements of the inverse problem of the reconstruction of the
boundary-value problem $L:$ from the Weyl function, from the spectral data
$\{\lambda_n, \gamma_n\}_{n\geq 0}$, and from two spectra $\{\lambda_n, \mu_n\}_{n\geq 0}$.

The Weyl $m$-function is defined by
\be\label{w1}
m(\lambda)= - \frac{\psi(0,\lambda)}{\Delta(\lambda)}.
\ee
By \eqref{010} and \eqref{p1} we obtain the asymptotic expansion
\be\label{w2}
m(\lambda) = \frac{1}{\sqrt{-\lam}} + O(\lambda^{-1})
\ee
along any ray except the positive real axis.

Let $\chi(x,\lambda)$ be a solution of \eqref{1} subject to the initial conditions
\[
\chi(0,\lambda)=0,\ \ \ \chi'(0,\lambda)=1
\]
and the jump conditions \eqref{3}. It is clear that $ W(\varphi,\chi) = 1\neq 0$ and the function $\psi(x,\lambda)$ can be
represented as
\be\label{b1}
\theta(x,\lambda):=\frac{\psi(x,\lambda)}{\Delta(\lambda)}=\chi(x,\lambda)-m(\lambda)\varphi(x,\lambda).
\ee
The functions $\theta(x,\lambda)$ and $m(\lambda)$ are called the Weyl solution and the Weyl function, respectively for the boundary value problem $L$. Clearly
\be\label{b3}
W(\varphi(x,\lambda),\theta(x,\lambda))=1.
\ee

\begin{lemma}\label{l41}
The Weyl function $m(\lambda)$ is a meromorphic Herglotz--Nevanlinna function,
\be
\im(m(\lambda)) = \im(\lambda) \|\theta(\lambda)\|_{\mathcal{H}}^2,
\ee
and can be represented as
\be\label{w3}
m(\lambda)=\sum_{n=0}^{\infty}\frac{\gamma_n}{\lambda_n-\lambda},
\ee
where
\be\label{w4}
\sum_{n=0}^{\infty} \frac{\gamma_n}{1+|\lam_n|^\gam} < \infty, \qquad \forall\gam>\frac{1}{2}.
\ee
\end{lemma}
\begin{proof}
The first relation follows after a straightforward calculation using
\be
\im(\theta(\pi,\lambda)\ol{\theta'(\pi,\lambda)})- \im(\theta(0,\lambda)\ol{\theta'(0,\lambda)}) = \im(\lambda) \int_0^\pi |\theta(x,\lambda)|^2 w(x) dx.
\ee
Hence $m(z)$ is a Herglotz--Nevanlinna function (i.e.\ it maps the upper half plane to the upper half plane) and by the
asymptotics \eqref{w2} it has a representation of the form (\cite[Lem.~9.20]{te})
\[
m(\lambda)= \int_\R \frac{d\rho(t)}{\lambda_n-t},
\]
where $\rho$ is a Borel measure satisfying
\[
\int_\R \frac{d\rho(t)}{1+|\lambda|^\gamma}, \qquad \forall\gam>\frac{1}{2}.
\]
Since by \eqref{w1} the Weyl function is meromorphic it follows that
$\rho$ is a pure point measure supported at the poles with masses given by the negative residues.
Hence the result follows from Lemma~\ref{lem:deldot}.
\end{proof}

Now we are ready to prove our main uniqueness theorem for the solutions of the problems \eqref{1}--\eqref{3}. For this purpose we agree that together with $L$ we consider a boundary value problem $\ti{L}$ of the same form but with different coefficients $\ti{q}(x)$, $\ti{h}$, $\ti{H}$, $\ti{a}_i$, $\ti{b}_i$, $\ti{c}_i$, $\ti{d}_i$. If a certain symbol $\eta$ denotes an object related to $L$, then $\ti \eta$ will denote the analogous object related to $\ti{L}$.

\begin{theorem}\label{t42}
If $m(\lambda)=\ti m(\lambda)$ and $w(x)=\ti{w}(x)$ then $L=\ti L$. Thus, the specification of the Weyl function and the weight function $w(x)$ uniquely determines the operator.
\end{theorem}

\begin{proof}
It follows from \eqref{p1} and \eqref{b1} that
\be\label{m1}
|\theta^{(\nu)}(x,\lambda)|\leq C |\rho|^{\nu-1}\exp(-|\tau|x),\qquad \nu=0,1,
\ee
as $\lambda\to\infty$ along any ray except the positive real axis.
Define the matrix $P(x, \lambda) = [P_{jk}(x, \lambda)]_{j,k=1,2}$ by the formula
\[
P(x, \lambda)
                 \begin{pmatrix}
                   \ti\varphi(x,\lambda) & \ti\theta(x,\lambda) \\
                   \ti\varphi'(x,\lambda) & \ti\theta'(x,\lambda) \\
                 \end{pmatrix}
=
                 \begin{pmatrix}
                   \varphi(x,\lambda) & \theta(x,\lambda) \\
                  \varphi'(x,\lambda) & \theta'(x,\lambda) \\
                 \end{pmatrix}.
\]
Taking \eqref{b3} into account we calculate
\be\label{17}
    \begin{pmatrix}
      P_{11}(x, \lambda) & P_{12}(x, \lambda) \\
      P_{21}(x, \lambda) & P_{22}(x, \lambda) \\
    \end{pmatrix}
=
                 \begin{pmatrix}
                   \varphi\ti\theta'-\ti\varphi'\theta&\ti\varphi\theta-\varphi\ti\theta\\
                   \varphi'\ti\theta'-\ti\varphi'\theta' & \ti\varphi\theta'-\varphi'\ti\theta \\
                 \end{pmatrix}
\ee
and
\be\label{18}
    \begin{pmatrix}
      \varphi(x, \lambda) \\
      \theta(x, \lambda)\\
    \end{pmatrix}
=
            \begin{pmatrix}
             P_{11}(x, \lambda)\ti\varphi(x, \lambda)+ P_{12}(x, \lambda)\ti\varphi'(x, \lambda)\\
              P_{11}(x, \lambda) \ti\theta(x, \lambda)+P_{12}(x, \lambda)\ti\theta'(x, \lambda)\\
            \end{pmatrix}.
\ee
It is easy to see that the functions $P_{jk}(x, \lambda)$, $j,k=1,2$, are meromorphic in $\lambda$ with simple poles in the points $\lambda_n$ and $\ti\lambda_n$. Moreover, if $m(\lambda)=\ti m(\lambda)$, then from \eqref{b1} and \eqref{17}, $P_{11}(x, \lambda)$ and $P_{12}(x, \lambda)$ are entire functions of growth order $1/2$ in $\lambda$. From \eqref{m1}
\be\label{m2}
| P_{11}(x, \lambda)|\leq C,\qquad | P_{12}(x, \lambda)|\leq \frac{C}{|\rho|}
\ee
along any ray except the positive real axis. Moreover, by our hypothesis this function has an order of growth $s$ and thus we
can apply the Phragm\'en--Lindel\"of theorem (e.g., \cite[Sect.~6.1]{levin}) the two half-planes bounded by the imaginary axis.
This shows that the functions $P_{11}$ and $P_{12}$ are bounded on all of $\C$ and thus constant by Liouville's theorem.
Since $P_{12}$ vanishes along a ray it must be zero and we obtain $P_{11}(x, \lambda)=A(x)$ and $P_{12}(x, \lambda)=0$.
Using \eqref{18}, we get
\be\label{19}
\varphi(x, \lambda)=A(x)\ti\varphi(x, \lambda),\ \theta(x, \lambda)=A(x)\ti\theta(x, \lambda).
\ee
It follows from \eqref{d3}, $W(\varphi(x, \lambda),\theta(x, \lambda))=W(\ti\varphi(x, \lambda),\ti\theta(x, \lambda))=1$ and so we deduce $A(x)=\frac{\ti{w}(x)}{w(x)}=1$, that is, $\varphi(x, \lambda)=\ti\varphi(x, \lambda)$, $\theta(x, \lambda)=\ti\theta(x, \lambda)$, and $\psi(x, \lambda)=\ti\psi(x, \lambda)$. Therefore from \eqref{1}, \eqref{3}, \eqref{d3}, and \eqref{0101} we get $q(x)=\ti q(x)$, a.e. on $[0,\pi]$ and $a_i=\ti a_i$, $b_i=\ti b_i$, $c_i=\ti c_i$, and $d_i=\ti d_i$ for $i=1,2,...,m-1$, $h=\ti h$ and $H=\ti H$ for $j=1,2,3$. Consequently $L=\ti L$.
\end{proof}

Note that this theorem is optimal in the sense that the weight function cannot be determined from $m(\lambda)$ since a unitary transformation as in Lemma~\ref{lem:unitary} can be used to change the weight without changing $m(\lambda)$. Note that the condition $w(x)=\ti{w}(x)$ will hold
if we have for example $a_ib_i=\ti a_i\ti b_i=1$ for all $i$ or $a_lb_l=\ti a_l\ti b_l$, and $d_l=\ti d_l$, for $l=1,2,...,m-1$.

By virtue of Lemma~\ref{l41} we also get:

\begin{corollary}\label{c43}
If $\lambda_n=\ti\lambda_n$ and $\gamma_n=\ti\gamma_n$, for $n=0,1,2,...$, and $w(x)=\ti{w}(x)$ then $L=\ti L$.
\end{corollary}

Finally, let us consider the boundary value problem $L^k$ which is the problem where the boundary condition $L_1(y)$ is
replaced by
\[
L_1'(y)= \begin{cases} y'(0) + k\, y(0)=0, & k \in \R,\\
y(0)=0, & k=\infty.\end{cases}
\]
Let $\{\mu_n\}_{n\geq 0}$ be the eigenvalues of the problem $L^k$.

\begin{corollary}
Suppose $k\ne h$.
If $\lambda_n=\ti\lambda_n$ and $\mu_n=\ti\mu_n$ for $n=0,1,2,...$, and $w(x)=\ti{w}(x)$, then $L=\ti L$.
\end{corollary}

\begin{proof}
We begin with the case $k=\infty$.
The numbers $\lambda_n$, $\mu_n$ are the poles and zeros of $m(\lambda)$ and hence determine it uniquely up
to a constant by Krein's theorem \cite[Thm.~27.2.1]{levin}. This unknown constant can be determined from \eqref{w2}.
The case $k\ne h$ follows in the same manner using $m(\lambda) + (k-h)^{-1}$.
\end{proof}

Finally, we are also able to extend Hald's theorem to the case of finitely many transmission conditions.

\begin{theorem}\label{thmHL}
If $\lambda_n=\ti\lambda_n$, $w(x)=\ti{w}(x)$, $L_1=\ti{L}_1$, $q(x)=\ti{q}(x)$ for a.e.\ $x<\frac{\pi}{2}$ and $U_i=\ti{U}_i$, $V_i=\ti{V}_i$ for all $i$ with $d_i< \frac{\pi}{2}$, then $L=\ti L$.
\end{theorem}

\begin{proof}
By the Hadamard factorization theorem
$W(\ti\varphi,\ti\psi)= C\, \ti{W}(\varphi,\psi)$ for some constant $C$ which can be determined from the asymptotic as $\lam\to\infty$:
\[
C= \prod_{i: d_i \ge \pi/2} \frac{\ti{\alpha}_i}{\alpha_i}>0.
\]
Furthermore, our assumptions imply
\[
\ti{\psi}(x,\lam) = C\, \psi(x,\lam) + F(\lam) \varphi(x,\lam), \quad x < \frac{\pi}{2},
\]
for some entire function $F(\lam)$ of growth order at most $\frac{1}{2}$. Solving for $F$ and taking the limit $x\uparrow\frac{\pi}{2}$ we obtain
\begin{align*}
 F(\lam) = \frac{\ti{\psi}(\frac{\pi}{2}-,\lam)- C \psi(\frac{\pi}{2}-,\lam)}{\varphi(\frac{\pi}{2}-,\lam)} = C \frac{\psi(\frac{\pi}{2}-,\lam)}{\varphi(\frac{\pi}{2}-,\lam)}\left(\frac{\ti{\psi}(\frac{\pi}{2}-,\lam)}{C \psi(\frac{\pi}{2}-,\lam)} -1\right).
\end{align*}
Now the expression in parenthesis vanishes along every ray different from the positive real axis while the expression in front remains bounded
by the asymptotics \eqref{12} for $\varphi$ and the analogous result for $\psi$, $\ti{\psi}$. Thus it must be identically zero by the
Phragm\'en--Lindel\"of theorem. Finally, $\ti{\chi}(x,\lam) = \chi(x,\lam)$ for $x< \frac{\pi}{2}$ implies that the associated Weyl functions
are equal and the claim follows from Theorem~\ref{t42}.
\end{proof}


\section{Uniqueness results for eigenparameter dependent boundary conditions}

In this last section we will replace the Robin boundary condition \eqref{2} by the following eigenparameter dependent boundary conditions
\begin{align}\label{2'}
  &L_1(y):=\lambda(y'(0)+h_1 y(0))-h_2y'(0)-h_3y(0)= 0, \nonumber\\
   & L_2(y):=\lambda(y'(\pi)+ H_1y(\pi))-H_2 y'(\pi)-H_3y(\pi) =0,
 \end{align}
where we assume that $h_j$, $H_j$, $j=1,2,3$, are real numbers, satisfying
\be
r_1:=h_3-h_1h_2>0 \quad\mbox{and}\quad r_2:=H_1H_2-H_3>0.
\ee
In order to obtain a self-adjoint problem we will use the following Hilbert space $\mathcal{H}:=L_2((0,\pi); w)\oplus \C^2$ with
inner product defined by
 \begin{align}
    &\spr{F}{G}_{\mathcal{H}}:=\int_0^\pi f\ol{g} w +\frac{w(0)}{r_1}f_1\ol{g_1}
+\frac{w(\pi)}{r_2} f_2\ol{g_2},
\quad F=\begin{pmatrix} f(x) \\ f_1\\  f_2 \end{pmatrix}, \:
G=\begin{pmatrix} g(x) \\ g_1\\ g_2 \end{pmatrix}.
 \end{align}
Again the associated norm will be denoted by
$\|F\|_{\mathcal{H}}=\spr{F}{F}_{\mathcal{H}}^{1/2}$.
Next we introduce
\[\begin{array}{cc}
    R_1(y):=y'(0)+h_1y(0),\ \ & R'_1(y):=h_2y'(0)+h_3y(0), \\
    \\
    R_2(y):=y'(\pi)+H_1y(\pi),\  \ & R'_2(y):=H_2y'(\pi)+H_3y(\pi).
  \end{array}
\]
In this Hilbert space we construct the operator
\be
    A:\mathcal{H}\rightarrow \mathcal{H}
\ee
with domain
\be
    \dom{A}=\left\{ F=\begin{pmatrix} f(x) \\f_1\\f_2\end{pmatrix} \left|
    \begin{array}{c}
         f, f'\in AC\big(\cup_0^{m-1}(d_i,d_{i+1})\big), \: \ell f\in L^2(0,\pi) \\
             U_i(f)=V_i(f)=0, \ f_1=R_1(f),\ f_2=R_2(f)
           \end{array}\right.\right\}
\ee
by
\[
AF=\begin{pmatrix}\ell f \\ R'_1(f)\\R'_2(f) \end{pmatrix}\quad \text{with}\ F=\begin{pmatrix}f(x) \\ R_1(f)\\ R_2(f)\end{pmatrix}\in \dom{A}.
\]
By construction, the eigenvalue problem for $A$,
\be
   AY=\lambda Y, \qquad Y:=\begin{pmatrix}y(x) \\ R_1(y)\\R_2(y) \end{pmatrix}\in \dom{A},
\ee
is equivalent to the eigenvalue problem \eqref{1}, \eqref{3}, and \eqref{2'} for $L$.
A straightforward calculation shows:

\begin{lemma}
The operator $A$ is symmetric.
\end{lemma}

In particular, the eigenvalues of $A$, and hence of $L$, are real and simple. To see that they are simple
it suffices to observe that the associated Cauchy problem \eqref{1}, \eqref{3} subject to the initial conditions
 $f(x_0\pm0)=f_0$, $f'(x_0\pm0)=f_1$ (with $x_0\in[0,1]$) has a unique solution.

Suppose that the functions $\varphi(x,\lambda)$ and $\psi(x,\lambda)$ are solutions of \eqref{1} under the initial conditions
\begin{align}\label{010'}
      \varphi(0,\lambda)=\lambda-h_2 ,\ \varphi'(0,\lambda)=h_3-\lambda h_1,
\end{align}
and
\begin{align}\label{0101'}
      \psi(\pi,\lambda)=H_2-\lambda,\ \psi'(\pi,\lambda)=\lambda H_1- H_3
\end{align}
as well as the jump conditions \eqref{3}, respectively.
Moreover, we set
 \be\label{d3'}
\Delta(\lambda):= W(\varphi(\lambda),\psi(\lambda)) = -w(\pi) L_2(\varphi(\lambda)).
\ee
Then $\Delta(\lambda)$ is an entire function whose roots $\lambda_n$ coincide with the eigenvalues of $L$.
Moreover, the eigenfunctions $\varphi_i(x,\lambda_n)$ and $\psi_i(x,\lambda_n)$
 associated with a certain eigenvalue $\lambda_n$, satisfy the relation $\psi_i(x,\lambda_n)=\beta_n \varphi_i(x,\lambda_n)$,
 where, by \eqref{010'},
\be
\beta_n=\frac{\psi'(0,\lambda_n)+h_1\psi(0,\lambda_n)}{r_1}.
\ee
We also define the norming constant by
\[
\gamma_n:=\|\Phi(x,\lambda_n)\|_{\mathcal{H}}^{-2}, \qquad
\Phi(x,\lambda)= \begin{pmatrix} \varphi(x,\lambda) \\ R_1(\varphi) \\ R_2(\varphi) \end{pmatrix}.
\]
Then it is straightforward to verify:

\begin{lemma}
All zeros $\lambda_n$ of $\Delta(\lambda)$ are simple and the derivative is given by
\be
\dot{\Delta}(\lambda_n)=-\gamma_n^{-1}\beta_n.
\ee
\end{lemma}

The same argument as for Theorem~\ref{thm:asym} shows:

\begin{theorem}
Let $\lambda=\rho^2$ and $\tau:=\im\rho$. For equation \eqref{1} with boundary conditions
 \eqref{2'} and jump conditions \eqref{3} as $|\lambda|\rightarrow\infty$, the following
 asymptotic formulas hold:
\begin{align}
 \varphi(x,\lambda)=
  \begin{cases}
           \rho^2\cos\rho x +O(\rho\exp(|\tau|x)),  \qquad 0\leq x<d_1,  \\
           \rho^2[ \alpha_1\cos\rho x +\alpha'_1\cos\rho(x-2d_1)]+O(\rho\exp(|\tau|x)), \qquad d_1<x<d_2,\\
           \rho^2[ \alpha_1\alpha_2\cos\rho x +\alpha'_1\alpha_2\cos\rho(x-2d_1)+\alpha_1\alpha'_2\cos\rho(x-2d_2)\\
   \quad +\alpha'_1\alpha'_2\cos\rho(x+2d_1-2d_2)]+O(\rho\exp(|\tau|x)), \qquad d_2<x<d_3,\\
   \quad \vdots\\
         \rho^2[ \alpha_1\alpha_2...\alpha_{m-1}\cos\rho x+\\
  \quad + \alpha'_1\alpha_2...\alpha_{m-1}\cos\rho(x-2d_1)+...\\
  \quad + \alpha_1\alpha_2...\alpha'_{m-1} \cos\rho(x-2d_{m-1})+\\
  \quad + \alpha'_1\alpha'_2\alpha_3...\alpha_{m-1}\cos\rho(x+2d_1-2d_2)+...\\
  \quad +\alpha_1...\alpha'_i...\alpha'_j...\alpha_{m-1}\cos\rho(x+2d_i-2d_j)\\
  \quad+ \alpha_1...\alpha'_i...\alpha'_j...\alpha'_k...\alpha_{m-1}\cos\rho(x-2d_i+2d_j-2d_k)+... \\
  \quad + \alpha'_1\alpha'_2...\alpha'_{m-1}\cos\rho(x+2(-1)^{m-1} d_1+2(-1)^{m-2}d_2+...-2d_{m-1})]\\
  \quad +O(\rho\exp(|\tau|x)), \qquad  d_{m-1}<x\leq\pi,
\end{cases}
\end{align}
and
\begin{align}
  \varphi'(x,\lambda)=
  \begin{cases}
         \rho^3[-\sin\rho x] +O(\rho^2\exp(|\tau|x)),  \qquad 0\leq x<d_1,\\
          \rho^3[-\alpha_1\sin\rho x -\alpha'_1\sin\rho(x-2d_1)]+O(\rho^2\exp(|\tau|x)), \qquad d_1<x<d_2,\\
          \rho^3[-\alpha_1\alpha_2\sin\rho x -\alpha'_1\alpha_2\sin\rho(x-2d_1)-\\
   \qquad -\alpha_1\alpha'_2\sin\rho(x-2d_2) -\alpha'_1\alpha'_2\sin\rho(x+2d_1-2d_2)]\\
   \qquad +O(\rho^2\exp(|\tau|x)), \qquad d_2<x<d_3,\\
                       \quad \vdots&\\
      \rho^3[ -\alpha_1\alpha_2...\alpha_{m-1}\sin\rho x-\alpha'_1\alpha_2...\alpha_{m-1}\sin\rho(x-2d_1)-...-\alpha_1\alpha_2...\alpha'_{m-1}\\
 \quad \sin\rho(x-2d_{m-1})-\alpha'_1\alpha'_2\alpha_3...\alpha_{m-1}\sin\rho(x+2d_1-2d_2)-...\\
\quad -\alpha_1...\alpha'_i...\alpha'_j...\alpha_{m-1}\sin\rho(x+2d_i-2d_j)\\
\quad -\alpha_1...\alpha'_i...\alpha'_j...\alpha'_k...\alpha_{m-1}\sin\rho(x-2d_i+2d_j-2d_k)+...\\
\quad -\alpha'_1\alpha'_2...\alpha'_{m-1}\sin\rho(x+2(-1)^{m-1}d_1+2(-1)^{m-2}d_2+...-2d_{m-1})]\\
\quad +O(\rho^2\exp(|\tau|x)), \qquad  d_{m-1}<x\leq\pi,
\end{cases}
\end{align}
where
\be
\alpha_i=\frac{a_i+b_i}{2} \ \ \text{and} \ \ \alpha'_i=\frac{a_i-b_i}{2},
\ee
 for $i=1,2,...,m-1$. The characteristic function satisfies
\begin{align}
\Delta(\lambda)=&\rho^5 w(\pi)[\alpha_1\alpha_2...\alpha_{m-1}\sin\rho \pi+\alpha'_1\alpha_2...\alpha_{m-1}\sin\rho(\pi-2d_1)+...+\alpha_1\alpha_2...\alpha'_{m-1} \nonumber\\ &\sin\rho(\pi-2d_{m-1})+\alpha'_1\alpha'_2\alpha_3...\alpha_{m-1}\sin\rho(\pi+2d_1-2d_2)+...\nonumber\\ &+\alpha_1...\alpha'_i...\alpha'_j...\alpha_{m-1}\sin\rho(\pi+2d_i-2d_j)\nonumber\\
&+\alpha_1...\alpha'_i...\alpha'_j...\alpha'_k...\alpha_{m-1}\sin\rho(\pi-2d_i+2d_j-2d_k)+...\nonumber\\
&+\alpha'_1\alpha'_2...\alpha'_{m-1}\sin\rho(\pi+2(-1)^{m-1}d_1+2(-1)^{m-2}d_2+...-2d_{m-1})]\nonumber\\
&+O(\rho^4\exp(|\tau|\pi)).
\end{align}
\end{theorem}

It follows from the above theorem that
\be
|\varphi^{(\nu)}(x,\lambda)|=O(|\rho|^{\nu+2}\exp(|\tau|x)),\ \ 0\leq x \leq \pi, \: \nu=0,1
\ee
and so by substituting $x$ with $\pi-x$ we get the asymptotic form of $\psi(x,\lambda)$ and $\psi'(x,\lambda)$.
In particular,
\be\label{p1'}
|\psi^{(\nu)}(x,\lambda)|=O(|\rho|^{\nu+2}\exp(|\tau|(\pi-x))),\ \ 0\leq x \leq \pi, \: \nu=0,1.
\ee

As a consequence of Valiron's theorem (\cite[Thm.~13.4]{levin}) we obtain:

\begin{theorem}
The eigenvalues $\lambda_n = \rho_n^2$ of the boundary value problem $L$ satisfy
\[
{\rho_n}= n+o(n)
\]
as $n\rightarrow\infty$.
\end{theorem}

The Weyl $m$-function is defined by
\be\label{w1'}
m(\lambda)= - \frac{R_1(\psi(\lambda))}{r_1\Delta(\lambda)} = -\frac{\psi'(0,\lambda)+h_1\psi(0,\lambda)}{r_1\Delta(\lambda)}.
\ee
By \eqref{010} and \eqref{p1'} we obtain the asymptotic expansion
\be\label{w2'}
m(\lambda) = -\frac{1}{r_1\lambda} + O(\lambda^{-3/2})
\ee
along any ray except the positive real axis.

Let $\chi(x,\lambda)$ be a solution of \eqref{1} subject to the initial conditions
\[
\chi(0,\lambda)=-\frac{1}{r_1},\ \ \ \chi'(0,\lambda)=\frac{h_1}{r_1}
\]
and the jump conditions \eqref{3}. It is clear that $ W(\varphi,\chi) = 1\neq 0$ and the function $\psi(x,\lambda)$ can be
represented as
\be\label{b1'}
\theta(x,\lambda):=\frac{\psi(x,\lambda)}{\Delta(\lambda)}=\chi(x,\lambda)-m(\lambda)\varphi(x,\lambda).
\ee
The functions $\theta(x,\lambda)$ and $m(\lambda)$ are called the Weyl solution and the Weyl function, respectively for the boundary value problem $L$. Clearly
\be
W(\varphi(x,\lambda),\theta(x,\lambda))=1.
\ee

\begin{lemma}\label{l41'}
The Weyl function $m(\lambda)$ is a meromorphic Herglotz--Nevanlinna function,
\be
\im(m(\lambda)) = \im(\lambda) \|\Theta(\lambda)\|_{\mathcal{H}}^2, \qquad
\Theta(x,\lambda) = \begin{pmatrix} \theta(x,\lambda) \\ R_1(\theta(\lambda)) \\ R_2(\theta(\lambda))\end{pmatrix},
\ee
and can be represented as
\be\label{w3'}
m(\lambda)=\sum_{n=0}^{\infty}\frac{\gamma_n}{\lambda_n-\lambda}.
\ee
where
\be\label{w4'}
\sum_{n=0}^{\infty} \gamma_n = \frac{1}{r_1}.
\ee
\end{lemma}
\begin{proof}
The first relation follows after a straightforward calculation using
\be
\im(\theta(\pi,\lambda)\ol{\theta'(\pi,\lambda)})- \im(\theta(0,\lambda)\ol{\theta'(0,\lambda)}) = \im(\lambda) \int_0^\pi |\theta(x,\lambda)|^2 w(x) dx.
\ee
The first part follows as in the proof of Lemma~\ref{l41}.
Computing the asymptotic of \eqref{w3'} and comparing with \eqref{w2'} shows \eqref{w4'}.
\end{proof}

Now we are ready to prove our main uniqueness theorem for the solutions of the problems \eqref{1}, \eqref{3}, and \eqref{2'}. For this purpose we agree that together with $L$ we consider a boundary value problem $\ti{L}$ of the same form but with different coefficients $\ti{q}(x)$, $\ti{h}_j$, $\ti{H}_j$, $\ti{a}_i$, $\ti{b}_i$, $\ti{c}_i$, $\ti{d}_i$.

\begin{theorem}\label{t56}
If $m(\lambda)=\ti m(\lambda)$ and $w(x)=\ti{w}(x)$ then $L=\ti L$. Thus, the specification of the Weyl function and the weight function $w(x)$ uniquely determines the operator.
\end{theorem}

\begin{corollary}
If $\lambda_n=\ti\lambda_n$ and $\gamma_n=\ti\gamma_n$, for $n=0,1,2,...$, and $w(x)=\ti{w}(x)$ then $L=\ti L$.
\end{corollary}

Finally, let us consider the boundary value problem $L'$, where we take the condition $y'(0)+h_1 y(0)= 0$ instead of the condition \eqref{2'} in $L$. Let $\{\mu_n\}_{n\geq 0}$ be the eigenvalues of the problem $L'$.
\begin{corollary}
If $\lambda_n=\ti\lambda_n$ and $\mu_n=\ti\mu_n$ for $n=0,1,2,...$, and $w(x)=\ti{w}(x)$, $r_1=\ti{r}_1$ then $L=\ti L$.
\end{corollary}

\begin{theorem}
If $\lambda_n=\ti\lambda_n$, $w(x)=\ti{w}(x)$, $L_1=\ti{L}_1$, $q(x)=\ti{q}(x)$ for a.e.\ $x<\frac{\pi}{2}$ and $U_i=\ti{U}_i$, $V_i=\ti{V}_i$ for all $i$ with $d_i< \frac{\pi}{2}$, then $L=\ti L$.
\end{theorem}

\bigskip
\noindent
{\bf Acknowledgments.}
One of us (M.S.)\ gratefully acknowledges the extraordinary hospitality
of the Faculty of Mathematics of the University of Vienna, Austria, where this paper was written.

\section*{Addendum}

In the introduction we attribute Theorem~\ref{t42} and its two corollaries in the special case for one transmission condition with determinant one (i.e., $w\equiv1$)
to \cite{am}. After publication of this paper we learned that this special case was already obtained earlier by Yurko in
\begin{itemize}
\item V. A.  Yurko, {\em Integral transforms connected with discontinuous boundary value problems}, Integral Transforms Spec. Funct. {\bf 10}, 141--164 (2000).
\end{itemize}
We regret we only learned about this paper after publication of our article.
In addition, we note that Yurko's paper also solves the inverse spectral problem in this context. 

\end{document}